\documentclass{amsart}

\usepackage{amsthm,amssymb,amsmath,graphicx}
\usepackage[english]{babel}

\begin{document}
\title{Isomorphisms between positive and negative $\beta$-transformations}
\author{Charlene Kalle}
\address{Mathematical Institute\\
Leiden University\\
Postbus 9512\\
2300 RA Leiden\\
the Netherlands}
 \email{kallecccj@math.leidenuniv.nl }


\maketitle

\begin{abstract}
We compare a piecewise linear map with constant slope $\beta>1$ and a piecewise linear map with constant slope $-\beta$. These maps are called the positive and negative $\beta$-transformations. We show that for a certain set of $\beta$'s, the multinacci numbers, there exists a measurable isomorphism between these two maps. We further show that for for all other values of $\beta$ between 1 and 2 the two maps cannot be isomorphic.  
\end{abstract}

\newtheorem{prop}{Proposition}[section]
\newtheorem{theorem}{Theorem}[section]
\newtheorem{lemma}{Lemma}[section]
\newtheorem{cor}{Corollary}[section]
\newtheorem{remark}{Remark}[section]
\theoremstyle{definition}
\newtheorem{defn}{Definition}[section]
\newtheorem{ex}{Example}[section]

\maketitle

\section{Introduction}

A classical problem in ergodic theory is to classify measure preserving transformations. Since the celebrated results of Ornstein in 1970  (\cite{orn70}) much research has been done on finding invariants for invertible maps. For non-invertible systems there are some methods for comparing one-sided Markov shifts on a finite alphabet (see \cite{AMT97}, \cite{CL00} and \cite{HR02}), but in general not much is known. In this article we are interested in isomorphisms between certain piecewise linear expanding interval maps, called positive and negative $\beta$-transformations. As they are not invertible and they cannot in general be described by a finite state Markov shift, these maps do not fit in the framework of \cite{AMT97}, \cite{CL00} and \cite{HR02}.

The $\beta$-transformation $Tx = \beta x$ (mod 1) with $1< \beta <2$ is a simple example of a chaotic dynamical system. There is a close link between the transformation $T$ and $\beta$-expansions of real numbers, i.e., expressions of the form $ x = \sum_{k=1}^{\infty} \frac{b_k}{\beta^k}$, where $b_k \in \{0,1\}$ for all $k \ge 1$. Initiated by the works of R\'enyi (\cite{Ren57}) and Parry (\cite{Par60}), this relation was extensively studied. Dynamical properties of the map can be used to study combinatorial properties of the expansions and vice versa. In the past few years this link has lead to research on negative $\beta$-transformations. Recent results in theoretical computer science have shown that negative $\beta$-expansions, i.e., expansions of the form $x = \sum_{k=1}^{\infty}\frac{b_k}{(-\beta)^k}$, have better robustness properties when used for encoding purposes (\cite{AHK08}). In \cite{IS09} Ito and Sadahiro studied the relation between negative $\beta$-expansions and a specific piecewise linear map with constant slope $-\beta$ defined by
\[ S_{IS}(x) = -\beta x - \Big\lfloor -\beta x + \frac{\beta}{\beta+1} \Big\rfloor,\]
where $\lfloor y \rfloor$ indicates the largest integer not exceeding $y$. This transformation maps the interval $\big[ -\frac{\beta}{\beta+1}, \frac{1}{\beta+1} \big)$ to itself. Ito and Sadahiro used this map to show that the positive and negative $\beta$-expansions share many properties. Since then, many others have found similarities between positive and negative $\beta$-expansions (see \cite{DMP11}, \cite{FL09}, \cite{MP11} and \cite{NS12} for example). Piecewise linear expanding interval maps with constant negative slope in general and the map $S_{IS}$ in particular have also generated a lot of attention. The papers \cite{DK11}, \cite{Gora09}, \cite{Hof11}, \cite{IS09} and \cite{LS11} deal, among other things, with invariant measures, transitivity and attractors of the maps in question.

It seems natural to ask exactly how different the positive and negative slope maps really are. In this paper we let $1 < \beta <2$ and we compare two maps: the positive and negative $\beta$-transformations $T=T_{\beta}$ and $S=S_{\beta}:[0,1] \to [0,1]$ given by
\begin{equation}\label{q:tands}
Tx = \left\{
\begin{array}{ll}
\beta x, & \text{if } x \in \big[0, \frac{1}{\beta} \big),\\
\\
\beta x -1, & \text{if } x \in \big[ \frac{1}{\beta}, 1 \big),
\end{array}
\right.
\quad \text{ and } \quad Sx = \left\{
\begin{array}{ll}
1 -\beta x, & \text {if } \big[0, \frac{1}{\beta}\big),\\
\\
2 - \beta x, & \text{if } \big( \frac{1}{\beta},1 \big].
\end{array}\right.
\end{equation}
The map $S$ is measurably isomorphic to the map $S_{IS}$ with isomorphism $x \mapsto \frac{1}{\beta-1}-x$. As in \cite{LS11} we choose to work with $S$ instead of $S_{IS}$ since S is defined on the unit interval, it has the same critical point as $T$ and, as for $T$, its dynamical behaviour depends mainly on the orbit of the point 1. Figure~\ref{f:goldenmean} shows both maps for a specific $\beta$. Note that for all $1 < \beta < 2$ the map $T$ has only one fixed point, while $S$ has two. Hence, $T$ and $S$ cannot be topologically conjugated, i.e., there are no homeomorphisms $\theta: [0,1] \to [0,1]$ that preserve the dynamics by $\theta \circ T = S \circ \theta$. That is why we study the possibility of $T$ and $S$ being measurably isomorphic.

This article is organised as follows. We first describe the dynamical properties of the maps $T$ and $S$ that can be found in the literature. In the two next sections we prove the two main results. Theorem~\ref{t:multinacci} states that for a specific set of $\beta$'s, the multinacci numbers, the maps $T$ and $S$ are measurably isomorphic. It is proved by showing that $T$ and $S$ are both isomorphic to the same Markov shift. The second main result, Theorem~\ref{t:noniso}, states that for all other values of $\beta$ the maps $T$ and $S$ are not isomorphic. Showing that two maps are not isomorphic involves showing that no map can qualify as an isomorphism. Therefore the proof of Theorem~\ref{t:noniso} needed a novel approach. We consider the number of preimages that points in $[0,1]$ can possibly have under $T^n$ and $S^n$ for an appropriate choice of $n\ge 1$. We then show that these numbers do not coincide on sets of positive measure. For the proof of Theorems~\ref{t:multinacci} and~\ref{t:noniso} we need a precise description of the dynamical behaviour of the iterates of $T$ and $S$. The key result is Lemma~\ref{l:nbonacciS}.

\vskip .2cm
{\bf Acknowledgement:} The author worked on this project in various locations. The research was started in the EU FP6 Marie Curie Research Training Network CODY (MRTN 2006 035651) at Warwick University, then continued under the research grant FWF S6913 at the University of Vienna and finally finished under the NWO Veni grant 639.031.140 at Leiden University. The author would like to thank a number people for fruitful discussions and valuable suggestions: Karma Dajani, Sebastian van Strien, Henk Bruin and Tom Kempton.

\section{The dynamical properties of $T$ and $S$}
Let $1 < \beta <2$ and let $T$ and $S$ be the transformations defined in (\ref{q:tands}). Let $\mathcal B$ be the Borel $\sigma$-algebra on $[0,1]$ and $\lambda$ the Lebesgue measure on $([0,1], \mathcal B)$. It is well known that there exists a unique invariant measure $\mu$ for $T$ equivalent to $\lambda$. This measure is ergodic and it is the unique measure of maximal entropy with entropy $h_{\mu}(T) = \log \beta$. For $S$ it is known that a unique invariant measure, absolutely continuous with respect to $\lambda$ exists and that this is measure is ergodic (see \cite{LasY73} and \cite{LY78}). We call this measure $\nu$. Combined results from \cite{Hof81} and \cite{HK82} give that $\nu$ is the unique measure of maximal entropy of $S$ with $h_{\nu}(S) = \log \beta = h_{\mu}(T)$. The dynamics of both maps are largely determined by the {\em orbits} of the point 1, i.e., by the sets $\{ T^n 1 \, : \, n \ge 0 \}$ and $\{ S^n 1 \, : \, n \ge 0 \}$.

The focus of this article is on the existence of isomorphisms between $T$ and $S$. In general we call two dynamical systems $(X, \mathcal X, m_X, F)$ and $(Y, \mathcal Y, m_Y, G)$ (or the two transformations $F$ and $G$)  {\em (measurably) isomorphic} if there are sets $X' \in \mathcal X$ and $Y' \in \mathcal Y$ with $m_X(X')=1=m_Y(Y')$ and if there exists a map $\phi: X' \to Y'$ with the following properties:
\begin{itemize}
\item[(i)] $\phi$ is a measurable bijection,
\item[(ii)] $\phi$ is measure preserving, i.e., for all $A \in \mathcal Y$, $m_X(\phi^{-1}(A)) = m_Y (A)$,
\item[(iii)] $\phi$ preserves the dynamics, i.e., $\phi \circ F = G \circ \phi$.
\end{itemize}

As was mentioned in the introduction, we will prove Theorem~\ref{t:multinacci} by constructing a suitable Markov shift. A one-sided Markov shift on a finite alphabet is defined as follows. Let $\{ 1, 2, \ldots, n\}$ be the set of states or the alphabet and let $P=(p_{ij})_{1\le i,j \le n}$ be a stochastic matrix. Define the set
\[ Z = \big\{ (x_n)_{n \ge 1} \in \{1, \ldots, n\}^{\mathbb N}\, : \, p_{x_nx_{n+1}}>0 \, \text{ for all } n \ge 1 \big\}.\]
Let $\sigma:Z \to Z$ denote the left shift, i.e., $\sigma x =y$ with $y_n=x_{n+1}$ for all $n \ge 1$. The pair $(Z,\sigma)$ is called the {\em symbolic system} for the matrix $P$. Note that a symbolic system can also be given by a (0,1)-matrix instead of a stochastic matrix. Let $\mathcal F$ be the $\sigma$-algebra on $X$ generated by the cylinders, i.e., the sets obtained by specifying the values of the first $k$ elements of the sequence, $k \ge 1$. Let $q = (q_i)_{1 \le i \le n}$ be a positive probability vector such that $qP=q$. The Markov measure $m_P$ associated to $P$ is defined on the cylinders by
\[ m_P \big( \{ x \in Z \, | \, x_1=a_1, \ldots, x_k=a_k \}  \big) = q_{a_1} p_{a_1a_2} \cdots p_{a_{k-1}a_k}.\]
The system $(Z, \mathcal F, m_P, \sigma)$ is called a one-sided {\em Markov shift}.

\section{The multinacci numbers}\label{s:multinacci}
In this section we show that the positive and negative $\beta$-transformations are measurably isomorphic when $\beta$ is a multinacci number. Multinacci numbers are specific Pisot numbers. A {\em Pisot number} is a real algebraic integer bigger than 1 with all its Galois conjugates less than 1 in modulus. For $n \ge 2$, the $n$-th {\em multinacci number} $\beta_n$ is the Pisot number with minimal polynomial
\begin{equation}\label{q:polynomial}
p_n(x)=x^n-x^{n-1}-\cdots - x-1.
\end{equation}
Then $\beta_2$ is the golden ratio and $\beta_3$ is usually referred to as the tribonacci number. It is well-known that the sequence $\{ \beta_n \}_{n \ge 2}$ is strictly increasing and $\lim_{n \to \infty} \beta_n=2$. We use the multinacci numbers to give a detailed description of the dynamical behaviour of the maps $T$ and $S$.

For the results in this section, we need results from \cite{BB85} on piecewise linear Markov maps. A transformation $F:[0,1] \to [0,1]$ is called a {\em piecewise linear Markov map} if there exists points $0 = a_0 < a_1 < \cdots < a_n=1$ such that for all $1 \le i \le n$, $F$ is a linear homeomorphism on the interval $(a_{i-1}, a_i)$ with
\[ \bigcup_{\ell=j}^{k-1} (a_{\ell},a_{\ell+1}) \subseteq F (a_{i-1}, a_i) \subseteq [a_j,a_k] \quad \text{for some }\, 0\le j <k \le 1.\]
The partition $\{ (0, a_1),(a_1, a_2), \ldots, (a_{n-1},1) \}$ of the interval $[0,1]$ given by $a_0, \ldots, a_n$ is called a {\em Markov partition} for $F$. Note that when we say {\em partition}, we mean a finite collections of sets with positive $\lambda$-measure that are pairwise disjoint and cover the whole interval $[0,1]$ upto a set of $\lambda$-measure zero.

The map $F$ is associated to a (0,1)-matrix $M$ that says which transitions between intervals $(a_{i-1},a_i)$ are possible. To obtain $M$, we label these intervals with a number: $(a_{i-1},a_i)=E_j$ for some $1 \le j \le n$. (We could of course simply take $j=i$, but later we want to have the opportunity to choose $i \neq j$). Then $M_{jk} =1$ if $E_k \subseteq F E_j$ and $M_{jk}=0$ otherwise. Let $(Z, \sigma)$ be the symbolic system associated to $M$. Below are some results from \cite{BB85} adapted to our setting.
\begin{itemize}
\item[(r1)] (Proposition 1 of \cite{BB85}) Suppose $M$ is an irreducible $n \times n$ (0,1)-matrix with spectral radius $\beta >1$ and of which the 1's in each row are not separated by 0's. Suppose that $T=T_{\beta}$ (or $S=S_{\beta}$) is a piecewise linear Markov map with matrix $M$. Then there is a countable set $C \subset [0,1]$ and a homeomorphism $\phi: Z \to [0,1]\backslash C$ such that $\phi \circ \sigma = T \circ \phi$ (or $S \circ \phi$).
\item[(r2)] (Proposition 2 of \cite{BB85}) Under the same conditions as in (r1), the symbolic system $(Z, \sigma)$ given by $M$ admits a unique maximal Markov measure $m$.
\item[(r3)] (Theorem 2 of \cite{BB85}) Under the same conditions as in (r1), the measure $p = m \circ \phi^{-1}$ is the unique probability measure for $T$ (or $S$) that is invariant for $T$ (or $S$), maximizes entropy and is equivalent to $\lambda$.
\end{itemize}

We will show that for $\beta = \beta_n$ both $T$ and $S$ are piecewise linear Markov maps with the same (0,1)-matrix $M$. Since $\mu$ and $\nu$ are the unique measures of maximal entropy for $T$ and $S$ respectively, Theorem 2 from \cite{BB85} then gives that $T$ and $S$ are isomorphic, if $M$ satisfies the conditions from (r1). As an illustration we first show how this works for $\beta_2$.

\subsection{The golden ratio}
In Figure~\ref{f:goldenmean} we see both maps $T$ and $S$ for $\beta_2=\frac{1+\sqrt 5}{2}$.
\begin{figure}[ht]
\centering{\includegraphics{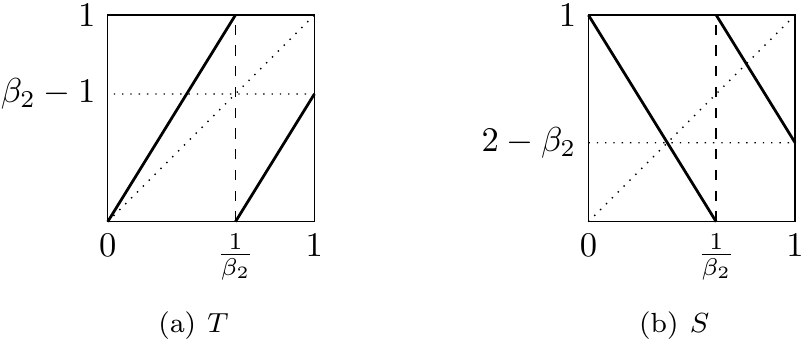}}
\vskip -.3cm
\caption{The maps $T$ and $S$ for the golden ratio $\beta_2 = \frac{1+\sqrt 5}{2}$.}
\label{f:goldenmean}
\end{figure}
Note that $T1=\beta_2-1=\frac{1}{\beta_2}$. Set $E_1 = \big(0, \frac{1}{\beta_2} \big)$ and $E_2 = \big( \frac{1}{\beta_2},1\big)$. Then $T E_1 = (0,1)$ and $T E_2 = E_1$. Hence, $T$ is a piecewise linear Markov map on the Markov partition $\{E_1,E_2\}$ with matrix
\[ M = \left(
\begin{array}{cc}
1&1\\
1&0
\end{array}
\right).\]

Let $(Z, \sigma)$ denote the symbolic system given by the matrix $M$. Note that $M$ is irreducible and that the 1's in each row are touching. Moreover, the characteristic polynomial of $M$ is equal to the minimal polynomial of $\beta_2$. Hence, the spectral radius of $M$ is $\beta_2$ and $T$ and $M$ satisfy all the conditions of (r1). So, $([0,1], \mathcal B, \mu, T)$ and $(Z, \mathcal F, m, \sigma)$ are isomorphic.

For $S$, note that $S \frac{1}{\beta_2^2} = 1- \frac{1}{\beta_2} =\frac{1}{\beta_2^2}$ and that
\[ S \Big(0, \frac{1}{\beta_2^2}\Big) = \Big(\frac{1}{\beta_2^2},1\Big) \quad \text{and} \quad S \Big(\frac{1}{\beta_2^2},1\Big) = \Big(0, \frac{1}{\beta_2^2} \Big) \cup \Big( \frac{1}{\beta^2_2}, 1\Big].\]
Hence, if we set $E_1=\big(\frac{1}{\beta_2^2},1\big)$ and $E_2=\big(0, \frac{1}{\beta_2^2}\big)$, then we see that $M$ is also the appropriate (0,1)-matrix for $S$. Hence, $([0,1], \mathcal B, \nu, S)$ is also isomorphic to $(Z, \mathcal F, m, \sigma)$ and therefore to $([0,1], \mathcal B, \mu, T)$.

\subsection{The maps $T$ and $S$ for $\beta_n$}
The map $T$ behaves very regularly for each multinacci number $\beta_n$. Note that for each $1 \le k \le n-2$,
\[ T^k 1 = \beta_n^k - \beta_n^{k-1} - \cdots - \beta_n -1 = \frac{1}{\beta_n} + \frac{1}{\beta_n^2} + \cdots + \frac{1}{\beta_n^{n-k}} > \frac{1}{\beta_n},\]
and hence
\begin{equation}\label{q:markovT}
0 < \frac{1}{\beta_n}=T^{n-1}1 < T^{n-2}1 < \cdots < T1 < 1.
\end{equation}
The points in (\ref{q:markovT}) give a Markov partition for $T$. Set $E_1 = (0, T^{n-1}1)$ and for $2 \le j \le n-1$, $E_j = (T^{n-j-1}1, T^{n-j}1)$. Then $T$ is a piecewise linear Markov map with associated (0,1)-matrix
\begin{equation}\label{q:mt}
M= \left(
\begin{array}{cccccc}
1 & 1 & 1 & \cdots & 1 & 1\\
1 & 0 & 0 & \cdots & 0 & 0\\
0 & 1 & 0 & \cdots & 0 & 0\\
\vdots & \vdots & \vdots & \ddots & \vdots & \vdots\\
0 & 0 & 0 & \cdots & 1 & 0
\end{array} \right).
\end{equation}
In Figure~\ref{f:graphbn} we see the corresponding graph.
\begin{figure}[ht]
\centering{\includegraphics{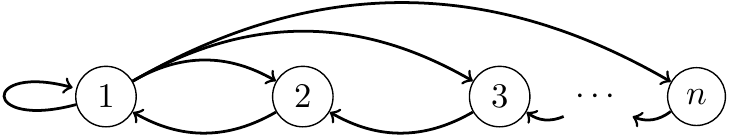}}
\caption{The graph corresponding to the dynamics of $T$ on the Markov partition from (\ref{q:markovT}).}
\label{f:graphbn}
\end{figure}

Let $(Z, \sigma)$ denote the symbolic system given by $M$. Note that $M$ is irreducible and that the 1's in each row are contiguous. Since the characteristic polynomial of $M$ is the minimal polynomial of $\beta_n$, the spectral radius of $M$ is $\beta_n$. Hence, all the requirements of (r1) are fulfilled and $([0,1], \mathcal B, \mu, T)$ and $(Z, \mathcal F, m, \sigma)$ are isomorphic.

\vskip .2cm
For the map $S$ we have to distinguish between even and odd $n$. Note that the small and the large fixed points of the map are given by $x_s = \frac{1}{\beta+1}$ and $x_{\ell}=\frac{2}{\beta+1}$ respectively. The following lemma on the orbit of 1 under $S$ is a key result in this paper.

\begin{lemma}\label{l:nbonacciS}
For all $n \ge 2$, if $\beta \ge \beta_n$ and $1 \le k \le n-2$, then 
\[ S^k 1 > x_{\ell} \,  \text{ if $k$ is even and } \, S^k 1 < x_s \, \text{ if $k$ is odd}.\]
Moreover,
\[ S^{n-1}1 \ge x_{\ell} \, \text{ if $n$ is odd and } \, S^{n-1}1 \le x_s \, \text{ if $n$ is even,}\]
with equality if and only if $\beta = \beta_n$.
\end{lemma}

\begin{proof}
To prove the lemma, we also prove the following claim.\\
{\bf Claim:} If $\beta \ge \beta_n$, then for all $1 \le k \le n-1$,
\begin{eqnarray*}
S^k 1 &=& \beta^k -2\beta^{k-1}+\beta^{k-2}-2\beta^{k-3} + \cdots + \beta^2 -2\beta +1 \quad \text{if $k$ is even, and}\\
S^k 1 &=& -\beta^k +2\beta^{k-1}-\beta^{k-2}+2\beta^{k-3} - \cdots - \beta +2 \quad \text{if $k$ is odd}.
\end{eqnarray*}

\vskip .2cm
We prove all statements by induction. It is clear that for all $\beta >1$, $S1 = - \beta +2 < \frac{1}{\beta}$. So the claim is clear for $n=2$. Then
\[ S1 = -\beta +2 \le x_s =\frac{1}{\beta+1} \quad \Leftrightarrow \quad \beta^2 - \beta -1\ge 0.\]
Recall from (\ref{q:polynomial}) that $p_2(x) = x^2-x-1$ is the minimal polynomial of $\beta_2$. Since $\beta_2$ is a Pisot number, we know that $\beta_2$ is the only root of this polynomial in the interval $(1,2)$. Hence $S1 \le x_s$ if and only if $\beta \ge \beta_2$ and equality holds if and only if $\beta = \beta_2$. This proves the lemma for $n=2$.

\vskip .2cm
For $n=3$, let $\beta > \beta_2$. Then $S^2 1 = 1 - \beta (-\beta +2) = \beta^2 - 2\beta +1$ and thus the claim is true for $n=3$. To prove the statement of the lemma,
\[ S^2 1 = \beta^2-2\beta +1 \ge x_{\ell} =\frac{2}{\beta+1} \quad \Leftrightarrow \quad \beta^3 - \beta^2 - \beta -1 \ge 0.\]
Note that the left hand side of the last inequality is $p_3(\beta)$, the minimal polynomial of the Pisot number $\beta_3$. Hence $S^21 \ge x_{\ell}$ if and only if $\beta \ge \beta_3$ with equality if and only if $\beta = \beta_3$.

\vskip .2cm
Now suppose that the lemma and the claim are true for some $n \ge 2$ and let $\beta > \beta_n$. Recall that the sequence $\{ \beta_n\}$ is strictly increasing towards 2, so that we only need to prove the claim and the last statement of the lemma for $S^n1$.

If $n$ is even, then by the induction hypothesis we know that
\[ S^{n-1}1 = -\beta^{n-1} +2\beta^{n-2} - \beta^{n-3} + \cdots + 2\beta^2- \beta +2 < x_s.\]
Hence,
\[S^n 1 = 1 - \beta S^{n-1}1 = \beta^n - 2\beta^{n-1} + \beta^{n-2} - \cdots - 2\beta^3 + \beta^2 - 2\beta+1,\]
which proves the claim in this case. Moreover, $S^n 1 \ge x_{\ell} = \frac{2}{\beta+1}$ if and only if
\[ \beta^{n+1} - \beta^n - \beta^{n-1} - \cdots - \beta -1 \ge 0.\]
Since the left hand side is the minimal polynomial of the Pisot number $\beta_{n+1}$, we have that $S^n 0 \le x_s$ if and only if $\beta \ge \beta_{n+1}$ with equality if and only if $\beta=\beta_{n+1}$.
 
Similarly, if $n$ is odd, then the induction hypothesis gives
\[S^{n-1}1 =  \beta^{n-1} - 2\beta^{n-2} +\beta^{n-3} - \cdots - 2\beta^3 + \beta^2 -2\beta+1 > x_{\ell},\]
and thus,
\[ S^n 1 = 2 -\beta S^{n-1}1 = -\beta^n - 2\beta^{n-1} + \beta^{n-2} - \cdots -2\beta^2 + \beta -2.\]
This shows that the claim is correct. Then, $S^n 1 \le x_s = \frac{1}{\beta+1}$ if and only if
\[ \beta^{n+1} - \beta^n - \beta^{n-1} - \cdots - \beta -1 \ge 0,\]
which proves the lemma as above.
\end{proof}

\begin{remark}\label{r:orbit0}{\rm
The map $S$ is linear with slope $-\beta$ on the intervals $[0,x_s]$ and $[x_{\ell},1]$. Therefore, Lemma~\ref{l:nbonacciS} gives us even more information on the orbit of 1 under $S$. For $n \ge 2$ and $\beta_{n-1} < \beta < \beta_n$, the order of the points $S^k 1$, $0 \le k \le n-1$, is completely determined as shown in Figure~\ref{f:points}. Recall that if $\beta=\beta_n$, then $S^{n-1} 1 =x_s$ or $x_{\ell}$ depending on the parity of $n$.}
\begin{figure}[ht]
\centering{\includegraphics{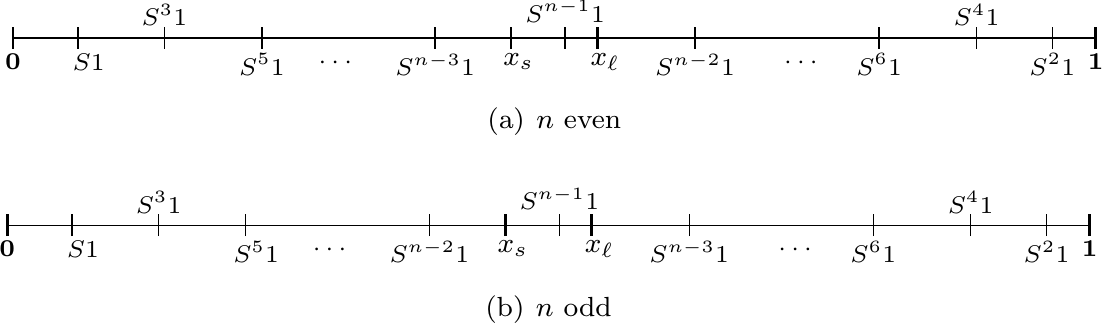}}
\vskip -.3cm
\caption{The relative position of the points $S^k 1$ for $\beta_{n-1} < \beta < \beta_n$.}
\label{f:points}
\end{figure}
\end{remark}

Remark~\ref{r:orbit0} and Figure~\ref{f:points} give us the Markov partition for $S$ when $\beta = \beta_n$. If $n$ is even, then $S^{n-1}1=x_s$ and $S$ is a piecewise linear Markov map with partition
\[ 0 < S1 < S^31 < \cdots < S^{n-3} 1 < S^{n-1}1=x_s < S^{n-2}1 < S^{n-4}1 < S^2 1 <1.\]
We label the partition elements as follows. Set $E_1=\big( x_s, S^{n-2}1 \big)$. Note that
\[ SE_1 = S \Big( \big( x_s, \frac{1}{\beta_n} \big) \cup \big[ \frac{1}{\beta_n}, S^{n-2}1 \big)  \Big) = \big(0, x_s \big) \cup \big( x_s,1\big].\]
Hence, all other states are mapped into $E_1$. Let $E_2 = \big( S^{n-3} 1,x_s \big)$ and note that then $SE_2=E_1$. Next set $E_3 = \big( S^{n-2}1, S^{n-4}1\big)$, so that $SE_3 = E_2$, and continue in this manner until we have $E_n = \big( 0,S1\big)$. Then $M$ from (\ref{q:mt}) is the (0,1)-matrix corresponding to $S$.

\vskip .2cm

If $n$ is odd, then the Markov partition is given by
\[ 0 < S1 < S^31 < \cdots < S^{n-2}1 < S^{n-1}1=x_{\ell} < S^{n-3}1 < S^{n-5}1 < \cdots < S^2 1 <1.\]
The labeling of the state goes in the same way as for the even case with $E_1=\big( S^{n-2}1, x_{\ell} \big)$. Then again $M$ from (\ref{q:mt}) is the (0,1)-matrix corresponding to the piecewise linear Markov map $S$. Since $M$ satisfies all the conditions from (r1), we have the following theorem.

\begin{theorem}\label{t:multinacci}
For $n \ge 2$, let $\beta_n$ denote the $n$-th multinacci number, i.e., the Pisot number with minimal polynomial $x^n-x^{n-1}-\cdots -x-1$. Then the positive $\beta_n$-transformation $T$ and the negative $\beta_n$-transformation $S$ are measurably isomorphic.
\end{theorem}

\section{Non-isomorphic transformations}\label{s:noniso}
Next we show that for all non-multinacci numbers $\beta$ the transformations $T$ and $S$ are {\bf not} isomorphic. We do this by studying the iterates of the maps $T$ and $S$. Recall that if $\phi$ is an isomorphism, then $\phi \circ T = S \circ \phi$ and thus $\phi \circ T^n = S^n \circ \phi$  for all $n \ge 1$. Hence, if we can show that for some iterate $T^n$ there is an interval (or set of positive $\mu$-measure), such that each point in this interval has $k$ preimages under $T^n$ and that the set of points with $k$ preimages under the map $S^n$ has $\nu$-measure zero, then $T$ and $S$ cannot be isomorphic. We make this more precise.

\vskip .2cm

For $n \ge 1$ define the sequences of maps $\psi^+_n, \psi^-_n : [0,1] \to \mathbb N$ by
\[ \psi^+_n (x) = \# \{ y \in [0,1] \, : \, T^n y =x \} \quad \text{and} \quad \psi^-_n(x) = \# \{ y \in [0,1] \, : \, S^n y=x \}. \]
These maps count the number of preimages that points have under the iterates of $T$ and $S$. For $k,n \ge 1$ define the sets
\[ I_n^+(k) = \{ x \in [0,1] \, : \, \psi_n^+(x)=k\} \quad \text{and} \quad I_n^-(k) = \{ x \in [0,1] \, : \, \psi_n^-(x)=k\}.\]
Note that the range of the functions $\psi_n^+$ and $\psi_n^-$ and the shapes of the sets $I_n^+(k)$ and $I_n^-(k)$ depend on the images of the intervals of monotonicity of $T^n$ and $S^n$ respectively. We call a non-empty, open interval $I$ an {\em interval of monotonicity} for $T^n$ if $T^n\big|_I$ is continuous and linear and if there is no bigger open interval $J \supset I$ with the same property. A similar definition holds for $S$.

Let $I$ be an interval of monotonicity of $T^n$. Then clearly $T^n I =(0, T^i 1)$ for some $0 \le i \le n$. Hence, the sets $I_n^+(k)$ are either empty or of the form $[T^i 1, T^j 1)$ for some $0 \le i, j \le n$, $i \neq j$. If $I$ is an interval of monotonicity for the map $S^n$, then $S^n I = (0, S^i 1)$ or $S^n I = (S^i1,S^j1)$ for some $0 \le i, j \le n$, $i \neq j$. Thus, the sets $I_n^-(k)$ are either empty or they contain an interval of the form $(S^i 1, S^j 1)$. If $\phi$ is an isomorphism from $T$ to $S$, then $\phi (I_n^+(k)) = I_n^-(k)$ and thus $\nu(I_n^-(k)) = \mu(I_n^+(k))$ for all $n,k \ge 1$. What we will do is show that if $\beta_n < \beta < \beta_{n+1}$, then there is a $k \ge 1$, such that $\mu(I_n^+(k)) \neq \nu(I_n^-(k))$. We do this by showing that one of the sets has measure zero, while the other has positive measure. To illustrate the general approach, we first consider $1 < \beta < \beta_2$ and $\beta_2 < \beta < \beta_3$.

\begin{figure}[ht]
\centering{\includegraphics{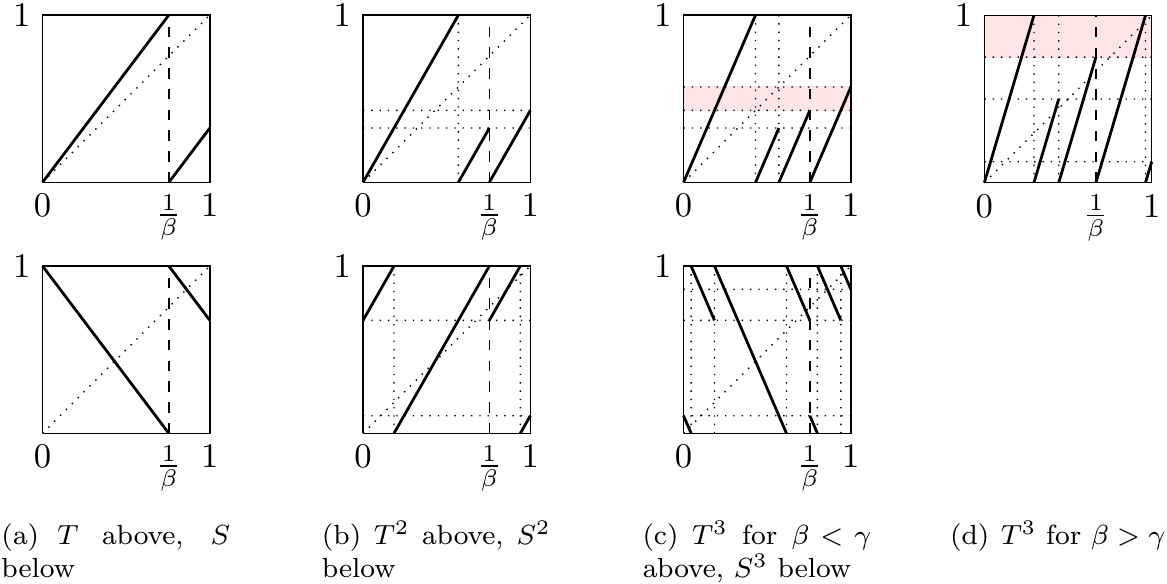}}
\caption{The first three iterates of $T$ and $S$ for some $1< \beta < \beta^2$. The coloured areas in (c) and (d) indicate the sets $I^+_3(2)$.}
\label{f:smallbetas}
\end{figure}

\subsection{The smallest $\beta$'s}
Let $1< \beta < \beta_2$. Then $T1=\beta -1 < \frac{1}{\beta}$. Moreover $T^2 1 =\beta^2-\beta < \frac{1}{\beta}$ if and only if $\beta < \gamma$, where $\gamma$ is the real root of the polynomial $x^3-x^2-1$. We see the first three iterates of the map $T$ in the top row of Figure~\ref{f:smallbetas}. Note that for $\beta < \gamma$, $I_3^+(2) = [T^21,T^31)$ and for $\beta \ge \gamma$, $I_3^+(2) = [T^21,1)$. Hence, $\mu(I^+_3(2))>0$ for all $1< \beta < \beta_2$.

Recall that $S$ has two fixed points and that $x_s=\frac{1}{\beta+1}$ is the smaller of the two. If $\beta < \beta_2$, then $x_s < S1 =2-\beta < \frac{1}{\beta}$. The first three iterates of $S$ are shown in the bottom row of Figure~\ref{f:smallbetas}. From Figure~\ref{f:smallbetas} we see that for all $1 < \beta < \beta_2$, $\nu(I^-_3(2)) = 0$. Hence, $\mu(I^+_3(2)) \neq \nu(I^-_3(2))$ and the maps $T$ and $S$ cannot be isomorphic.

\subsection{$\beta$'s between the golden ratio and the tribonacci number}
If $\beta_2 < \beta < \beta_3$, then $T1 = \beta-1>\frac{1}{\beta}$, $T^21 < \frac{1}{\beta}$ and $T^2 1 < T^3 1 < 1$. This allows us to calculate the range of $\psi_n^+$ for $n=1,2,3$. For $S$, Lemma~\ref{l:nbonacciS} implies that $S1=2-\beta < x_s$ and $x_s < S^2 1 < x_{\ell}$. To accurately describe the first three iterates of $S$, we have to know if $S^2 1 = 1-2\beta+\beta^2 > \frac{1}{\beta}$ or not. Therefore, let $\eta$ denote the real root of the polynomial $x^3-2x^2+x-1$. Then $S^2 1 < \frac{1}{\beta}$ if $\beta_2< \beta <\eta$ and $S^2 1 \ge \frac{1}{\beta}$ if $\eta \le \beta < \beta_3$. This allows us to calculate $\psi_n^-(x)$ for $n=1,2,3$ and all $x \in [0,1)$. In Figure~\ref{f:biggerbetas} we show the graphs of $T^3$ and $S^3$. There we see that for all $\beta_2 < \beta < \beta_3$, $\mu(I_3^+(3))>0$ and $\mu(I_3^+(6))>0$. These are the coloured areas in the top line of Figure~\ref{f:biggerbetas}. On the other hand, for $S^3$ we have $\nu(I_3^-(6)) =0$ if $\beta \le \eta$ (see Figure~\ref{f:biggerbetas}(d) and (e)) and $\nu(I_3^-(3)) =0$ if $\beta \ge \eta$ (see Figure~\ref{f:biggerbetas}(e) and (f)). Hence, $T$ and $S$ cannot be isomorphic for these $\beta$'s.

\begin{figure}[ht]
\centering{\includegraphics{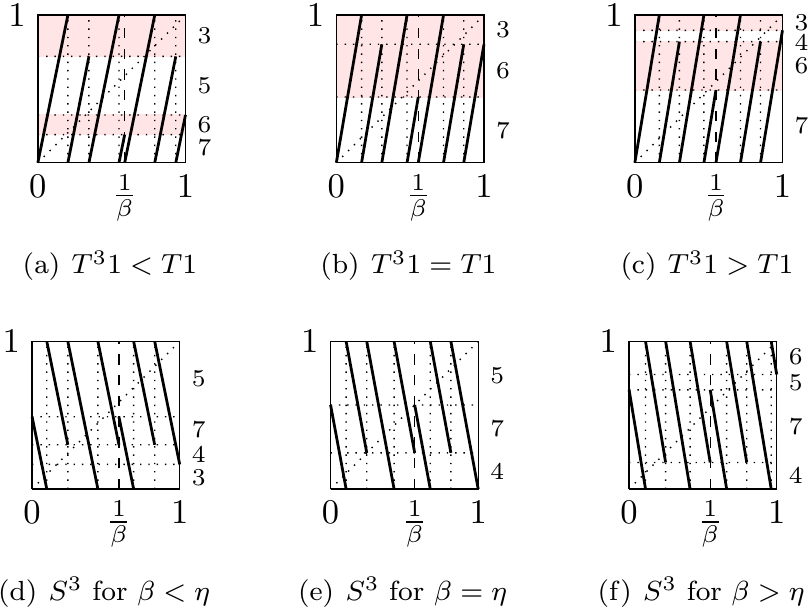}}
\caption{The transformations $T^3$ and $S^3$ for $\beta_2 < \beta < \beta_3$. The numbers to the right of the graphs give the values of the functions $\psi_3^+$ and $\psi_3^-$ on the corresponding intervals. The red areas indicate the intervals where points have three or six preimages for $T^3$.}
\label{f:biggerbetas}
\end{figure}

\subsection{Between $\beta_{n-1}$ and $\beta_n$}
To prove that $T$ and $S$ cannot be isomorphic if $\beta_{n-1}< \beta < \beta_n$, we follow the approach from the previous section. We will look at the ranges of $\psi^+_n$ and $\psi^-_n$ and therefore we study $T^n$ and $S^n$. We begin with $T$.

\subsubsection{The positive $\beta$-transformation between $\beta_{n-1}$ and $\beta_n$}
For $\beta_{n-1} < \beta < \beta_n$, $n \ge 4$, we know that
\begin{equation}\label{q:pointsT}
0< T^{n-1} 1 < \frac{1}{\beta} < T^{n-2} 1 < T^{n-3} 1 < \cdots < T1 < 1
\end{equation}
and $T^{n-1}1 < T^n 1 < 1$. We are interested in the range of $\psi_n^+$. To determine this exactly, we will take a closer look at the images of the (open) intervals of monotonicity of the iterates of $T$ and we will assign a type to each of them. The map $T$ has two intervals of monotonicity: $\big(0, \frac{1}{\beta} \big)$ and $\big( \frac{1}{\beta},1\big)$. The images of these intervals are $(0,1)$ and $(0, T1)$, hence there are two different types of images and both types occur exactly once for $T$. We call an interval of monotonicity with image $(0,1)$ type 0 and the one with image $(0,T1)$ type 1. This gives that
\[ \phi^+_1 (x) = \left\{
\begin{array}{ll}
2, & \text{if } x \in (0, T1) \subset I^+_1(2),\\
1, & \text{if } x \in (T1,1) \subset I^+_1(1).
\end{array}
\right.\]
Since we are only interested in the $\mu$-measure of sets, it is enough to consider open intervals here. $T^2$ has four intervals of monotonicity, but only three different types:
\[ \begin{array}{ll}
\big( 0, \frac{1}{\beta^2} \big) \, \text{with } T^2 \big( 0, \frac{1}{\beta^2} \big) = (0,1), &  \big( \frac{1}{\beta}, \frac{1}{\beta} + \frac{1}{\beta^2} \big) \, \text{with } T^2 \big( \frac{1}{\beta}, \frac{1}{\beta} + \frac{1}{\beta^2} \big) = (0,1),\\
\\
\big( \frac{1}{\beta^2}, \frac{1}{\beta} \big) \, \text{with } T^2 \big( \frac{1}{\beta^2}, \frac{1}{\beta} \big) = (0, T1), &  \big( \frac{1}{\beta}+\frac{1}{\beta^2}, 1 \big) \, \text{with } T^2 \big( \frac{1}{\beta}+\frac{1}{\beta^2}, 1 \big) = (0, T^21).
\end{array}\]
There are two times type 0, one type 1 and one new type, $\big(0, T^2 1\big)$, which we call type 2. Moreover, we see that a type 0 interval splits into a type 0 and a type 1 for the next iteration and a type 1 interval splits in a type 0 and a type 2. This gives
\[ \psi^+_2(x) = \left\{
\begin{array}{ll}
4, & \text{if } x \in (0, T^21) \subset I^+_2(4),\\
3, & \text{if } x \in (T^2 1, T1) \subset I^+_2(3),\\
2, & \text{if } x \in (T1,1) \subset I^+_2(2).
\end{array}
\right.\]
Going one step further, we see that for $T^3$ the type 2 interval splits into a type 0 interval and a type 3 interval with image $\big(0, T^3 1\big)$. Since $\frac{1}{\beta} < T^{n-2} 1 < T^{n-3} 1 < \cdots < T1 < 1$ this pattern continues up to $T^{n-1}$ and we can recursively determine the numbers with which each of the types of intervals occur for $T^{n-1}$. Let $\kappa_j(m)$ denote the number of type $j$ intervals that $T^m$ has. We have the following lemma.

\begin{lemma}\label{l:kappas}
For all $n \ge 4$, $1 \le m \le n-1$ and $0 \le j \le m-1$, $\kappa_j(m)=2^{m-1-j}$ and $\kappa_m(m)=1$. For all other values of $j$, $\kappa_j(m)=0$. Moreover, $\kappa_0(n) = 2^{n-1}-1$, $\kappa_n(n)=1$ and for $1 \le j \le n-1$, $\kappa_j(n)=2^{n-1-j}$.
\end{lemma}

\begin{proof}
Fix some $n \ge 4$. First note that for $1 \le m \le n-1$, $T^m$ can only have intervals of type $0, 1, \ldots, m$, which proves the middle part of the lemma. We prove the first statement by induction on $m$. We have already proved the statement for $m=1,2$. Now suppose that the statement is true for some $m <n-1$. Each type $j$ interval, $0 \le j \le m$, has image $(0, T^j 1)$ under $T^m$. Hence, by (\ref{q:pointsT}) each type $j$ interval splits into a type 0 interval and a type $(j+1)$ interval for $T^{m+1}$. Thus,
\[ \kappa_0(m+1) = \sum_{j=0}^m \kappa_j(m) = 1+ \sum_{j=0}^{m-1} 2^j = 2^m.\]
Moreover, $\kappa_j(m+1) = \kappa_{j-1} (m)$ for all $1 \le j \le m$ and the first statement of the lemma is proved.

For the last part, first note that $T^{n-1}1 < \frac{1}{\beta}$ implies that the type $(n-1)$ interval of $T^{n-1}$ does not split into two intervals of monotonicity for $T^n$. Hence, the type $(n-1)$ interval only produces a type $n$ interval for $T^n$. This implies that
\[ \kappa_0(n) = \sum_{j=0}^{n-2} \kappa_j(n-1) = \sum_{j=0}^{n-2} 2^j = 2^{n-1}-1.\]
For all $1 \le j \le n$ we still have $\kappa_j(n)=\kappa_{j-1}(n-1)$, which gives the lemma.
\end{proof}

The points $T^k 1$ partition the interval $[0,1]$ into intervals on which $\psi_n^+$ is constant. By adding the appropriate numbers $\kappa_i(n)$ we get the values of $\psi^+_n$. Note that $\psi^+_n$ is a monotone decreasing function on $[0,1)$. The next lemma gives some information about the range of $\psi^+_n$.

\begin{figure}[ht]
\centering{\includegraphics{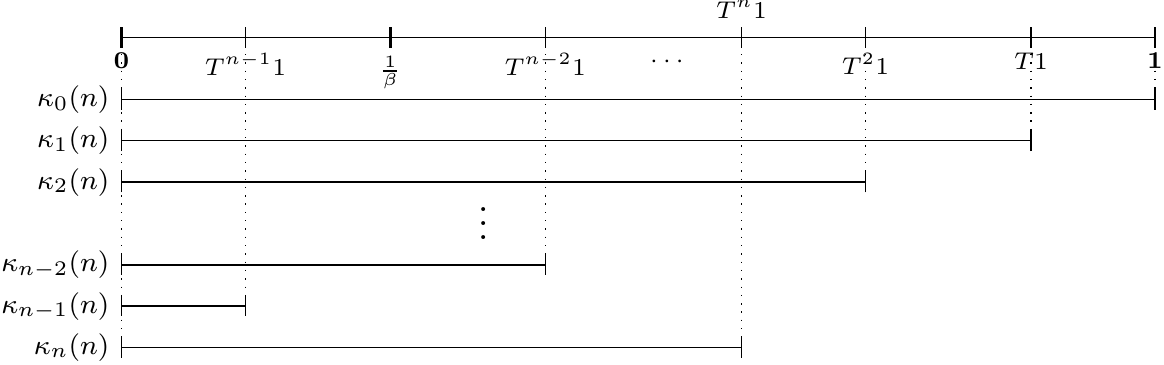}}
\caption{The images of the different types of intervals of monotonicity of $T^n$. We need to add the numbers $\kappa_j(n)$ for the right $j$ to determine $\psi_n^+$.}
\label{f:psi+}
\end{figure}

\begin{lemma}\label{l:values}
Let $n \ge 4$. Then $\psi^+_n(x)$ is even for $T^{n-1}1 < x <T^n 1$ and $\psi_n^+(x)$ is odd for $x \in (0, T^{n-1}1) \cup (T^n 1,1)$.
\end{lemma}

\begin{proof}
Let $1\le m \le n-1$ be such that $T^m 1 < T^n 1 < T^{m-1}1$. In Figure~\ref{f:psi+} we can see which numbers $\kappa_i(n)$ we need to add in order to get the values $\psi^+$ can take on each interval. For all $1 \le j \le m-1$ on the interval $(T^j1, T^{j-1}1)$ we have $\psi_n^+(x)=\sum_{i=0}^{j-1}\kappa_i(n)$. Also, $\psi_n^+(x)=\sum_{i=0}^{m-1}\kappa_i(n)$ on the interval $(T^n1, T^{m-1}1)$. The last part of Lemma~\ref{l:kappas} implies that these numbers are odd. Furthermore, $\psi_n^+(x)= \sum_{i=0}^{m-1}\kappa_i(n)+1$ on $(T^m1, T^n1)$, which is even, and for $m+1 \le j \le n-1$ we have the even value $\psi_n^+(x) = \sum_{i=0}^{j-1}\kappa_i(n)+1$ on $(T^j1, T^{j-1}1)$. On $(0, T^{n-1})$ we have $\psi_n^+(x) = \sum_{i=0}^n \kappa_i(n) = 2^n-1$, which is odd.
\end{proof}

\begin{remark}\label{r:values}{\rm
(i) By Lemma~\ref{l:values} we know a few values of $\psi_n^+$ explicitly. For example, $\psi_n^+(x) = 2^n-1$ on $(0, T^{n-1}1)$, which also the maximum of $\psi_n^+$. Since $T^{n-1}1 < T^n1$ we also know that $\psi_n^+(x) = 2^n-2$ on $(T^{n-1}1, \min\{T^{n-2}1,T^n1\})$. The minimum of $\psi_n^+$ is attained on the interval $(\max\{T^21, T^n1\},1)$ and here $\psi_n^+(x) = 2^{n-1}-1$. Since the exact position of $T^n 1$ is unknown, we cannot say more.\\
(ii) Lemma~\ref{l:values} also implies that all the odd values in the range of $\psi^+_n$ other than the maximum are smaller than the even values.
}\end{remark}

\subsubsection{The negative $\beta$-transformation between $\beta_{n-1}$ and $\beta_n$}
\noindent For the map $S$ the situation is more complicated. We will describe the structure of the images of the intervals of monotonicity of the iterates of $S$, just as we did for $T$. Recall that Figure~\ref{f:points} gives the relative positions of the points $S^k 1$. Note that the point $S^{n-1}1$ can be either to the left or to the right of $\frac{1}{\beta}$. This means that for all $\beta_{n-1} < \beta < \beta_n$ the maps $S^k$, $1\le k \le n-1$ have the same structure and the map $S^n$ can behave in two different ways, according to where the point $S^{n-1}1$ lies relative to $\frac{1}{\beta}$.

\vskip .2cm

Consider the maps $S^k$, $1 \le k \le n-1$. As before we assign a type to each of the (open) intervals of monotonicity of $S^k$ according to the shape of their images under $S^k$. We are interested in determining the numbers of intervals of each type. For $S$, we have two types, $S\big(0,\frac{1}{\beta}\big) = (0,1)$ and $S\big(\frac{1}{\beta},1\big) = (S1,1)$, and both occur once. We call the one with image $(0,1)$ type 0 and the other one type 1. Then,
\[ \psi^-_1(x) = \left\{
\begin{array}{ll}
2, &\text{if } x \in (S1,1) \subset I^-_1(2),\\
1, & \text{if }x \in (0, S1) \subset I^-_1(1).
\end{array}
\right.\]
Again we are not bothered with the endpoints of the intervals, since we are only interested in the measures of the sets. For $S^2$ the intervals of monotonicity are $\big(0, \frac{1}{\beta}-\frac{1}{\beta^2} \big)$, $\big( \frac{1}{\beta}-\frac{1}{\beta^2}, \frac{1}{\beta} \big)$, $\big( \frac{1}{\beta}, \frac{2}{\beta}-\frac{1}{\beta^2} \big)$ and $\big( \frac{2}{\beta}-\frac{1}{\beta^2}, 1\big)$.  Then
\begin{align*}
S^2 &\Big( 0, \frac{1}{\beta}-\frac{1}{\beta^2} \Big)= S^2 \Big( \frac{1}{\beta}, \frac{2}{\beta}-\frac{1}{\beta^2} \Big) = (S1,1),\\
S^2 &\Big( \frac{1}{\beta}-\frac{1}{\beta^2}, \frac{1}{\beta}\Big) = (0,1), \quad S^2 \Big( \frac{2}{\beta}-\frac{1}{\beta^2},1 \Big) = (0, S^21).
\end{align*}
So, the type 0 interval for $S$ splits into a type 0 and a type 1 interval for $S^2$ and the type 1 interval for $S$ splits into a type 1 interval and an interval with image $\big(0, S^21\big)$, which we call type 2 for $S^2$. Hence, for $S^2$ we have one type 0, two times type 1 and one type 1. This gives
\[ \psi^-_2(x) = \left\{
\begin{array}{ll}
4, & \text{if } x \in (S1,S^21) \subset I^-_2(4),\\
3, & \text{if } x \in (S^2 1,1) \subset I^-_2(3),\\
2, & \text{if } x \in (0, S1) \subset I^-_2(2).
\end{array}
\right.\]
For $S^3$ we only have to determine how the type 2 interval splits. Since $S^21 > x_{\ell} > \frac{1}{\beta}$, the interval becomes a type 0 interval and a type 3 interval with image $\big(S^3 1,1 \big)$. If $n \ge 5$, then $S^3 1 < x_s < \frac{1}{\beta}$ and the type 3 interval splits into a type 1 and a type 4 for $S^4$. We can continue in this manner upto $S^{n-1}$, see Figure~\ref{f:types}.
\begin{figure}[ht]
\centering{\includegraphics{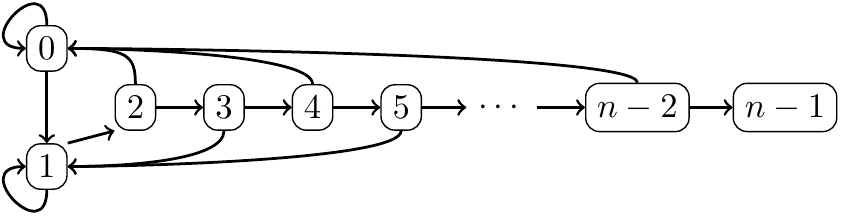}}
\vskip -.5cm
\caption{The subdivisions of the interval types for $S^{n-1}$, $\beta_{n-1} < \beta < \beta_n$ and $n$ even. The numbers in the nodes are the interval types. For odd values of $n$, the arrow starting from the node labelled ($n-2$) would end in node 1 instead of 0.}
\label{f:types}
\end{figure}

We can now inductively determine the number of intervals of monotonicity of each type for $S^{n-1}$ and then also for $S^n$. Let $\iota_j(m)$ denote the number of intervals of type $j$ that $S^m$ has. Hence, we know that $\iota_0(1)=\iota_1(1)=1$ and $\iota_0(2)=\iota_2(2)=1$ and $\iota_1(2)=2$.
\begin{lemma}\label{l:iotas}
For all $n \ge 4$ and $3 \le m \le n-1$,
\[ \iota_0(m) = 2^{m-2}, \quad \iota_{m-1}(m)=2, \quad \iota_m(m)=1,\]
and for all $1 \le j \le m-2$, $\iota_j(m)=3 \cdot 2^{m-2-j}$.
\end{lemma} 

\begin{proof}
Fix $n \ge 4$. Note that for all $3 \le m \le n-1$, we have
\begin{eqnarray}
\iota_0(m) &=& \sum_{\stackrel{0 \le j \le m-1:}{j \text{ even}}} \iota_j(m-1), \quad \iota_1(m) = \iota_0(m-1) + \sum_{\stackrel{0 \le j \le m-1:}{j \text{ odd}}} \iota_j(m-1),\label{q:rec1}\\
\iota_j(m) &= & \iota_{j-1}(m-1), \quad 2 \le  j \le m.\label{q:rec2}
\end{eqnarray}
We prove the lemma by induction. By (\ref{q:rec1}) and (\ref{q:rec2}) we have for $m=3$ that $\iota_0(3)=\iota_2(3)=2$, $\iota_1(3)=3$ and $\iota_3(3)=1$.

Now, suppose that the lemma is true for some odd $3 \le m \le n-2$. Then, for $m+1$ the induction hypothesis gives
\[ \iota_0(m+1) = \sum_{\stackrel{0 \le j \le m:}{j \text{ even}}} \iota_j(m)
= 2^{m-2} + 3\cdot 2^{m-4} + \cdots + 3\cdot 2 + 2 = \sum_{i=0}^{m-2} 2^i +1 = 2^{m-1}.\]
Also,
\begin{eqnarray*}
\iota_1(m+1) &=& \iota_0(m) + \sum_{\stackrel{0 \le j \le m:}{j \text{ odd}}} \iota_j(m) = 2^{m-2} + 1 + \sum_{\stackrel{0 \le i \le m-3:}{i \text{ even}}} 3\cdot 2^i\\
&=& 2^{m-1} + 2^{m-2} = 3\cdot 2^{m-2}.
\end{eqnarray*}
For $2 \le j \le m+1$ the result follows from the fact that $\iota_j(m+1)=\iota_{j-1}(m)$.

If $m$ is even, then
\[ \iota_0(m+1) = 2^{m-2} + 3\cdot 2^{m-4} + \cdots + 3 + 1\]
and
\[ \iota_1(m+1) = 2^{m-2} + 3\cdot 2^{m-3} + 3\cdot 2^{m-5} + \cdots + 3\cdot 2 +2.\]
The rest of the proof is exactly the same.
\end{proof}

As for $T$ we are interested in $\psi^-_n$. Therefore, we also need to know $\iota_j(n)$ for all $0 \le j \le n$. These values depend on the position of $S^{n-1}1$, see Figure~\ref{f:psi-}.

\begin{lemma}\label{l:cases}
Let $n \ge 4$ be an even number.
\begin{itemize}
\item[(1)] If $S^{n-1}1 < \frac{1}{\beta}$, then $\iota_0(n) = 2^{n-2}$, $\iota_{n-1}(n)=2$, $\iota_n(n)=1$ and for all $1 \le j \le n-2$, $\iota_j(n)=3 \cdot 2^{n-2-j}$. Moreover, the type $n$ interval has image $(0,S^n1)$ under $S^n$.
\item[(2)] If $S^{n-1}1 = \frac{1}{\beta}$, then $\iota_0(n) = 2^{n-2}$, $\iota_{n-1}(n)=2$, $\iota_n(n)=0$ and for all $1 \le j \le n-2$, $\iota_j(n)=3 \cdot 2^{n-2-j}$.
\item[(3)] If $S^{n-1}1 > \frac{1}{\beta}$, then $\iota_0(n) = 2^{n-2}$, $\iota_1(n)=3 \cdot 2^{n-3}-1$, $\iota_{n-1}(n)=2$, $\iota_n(n)=1$ and for all $2 \le j \le n-2$, $\iota_j(n)=3 \cdot 2^{n-2-j}$. Moreover, the type $n$ interval has image $(S1,S^n1)$ under $S^n$.
\end{itemize}

\noindent If on the other hand $n \ge 4$ is odd, then we have the following.
\begin{itemize}
\item[(1*)] If $S^{n-1}1 < \frac{1}{\beta}$, then $\iota_0(n) = 2^{n-2}-1$, $\iota_{n-1}(n)=2$, $\iota_n(n)=1$ and for all $1 \le j \le n-2$, $\iota_j(n)=3 \cdot 2^{n-2-j}$. Moreover, the type $n$ interval has image $(S^n1,1)$ under $S^n$.
\item[(2*)] If $S^{n-1}1 = \frac{1}{\beta}$, then $\iota_0(n) = 2^{n-2}$, $\iota_{n-1}(n)=2$, $\iota_n(n)=0$ and for all $1 \le j \le n-2$, $\iota_j(n)=3 \cdot 2^{n-2-j}$.
\item[(3*)] If $S^{n-1}1 > \frac{1}{\beta}$, then $\iota_0(n) = 2^{n-2}$, $\iota_{n-1}(n)=2$, $\iota_n(n)=1$ and for all $1 \le j \le n-2$, $\iota_j(n)=3 \cdot 2^{n-2-j}$. Moreover, the type $n$ interval has image $(S^n1,1)$ under $S^n$.
\end{itemize}
\end{lemma}

\begin{proof}
By Lemma~\ref{l:iotas} we know that $\iota_0(n-1) = 2^{n-3}$, $\iota_{n-2}(n-1)=2$, $\iota_{n-1}(n-1)=1$ and for all $1 \le j \le m-2$, $\iota_j(n-1)=3 \cdot 2^{n-3-j}$. First assume that $n$ is even. Recall that the image of the type ($n-1$) interval under $S^{n-1}$ is $(S^{n-1}1,1)$. We treat each case separately.

\vskip .2cm
(1) Suppose $S^{n-1}1 < \frac{1}{\beta}$. Then the type ($n-1$) interval splits into an interval with image $(S1,1)$, i.e., a type 1 interval, and a type $n$ interval with image $(0,S^n1)$ under $S^n$. Hence, the recursions (\ref{q:rec1}) and (\ref{q:rec2}) still hold and this gives the result in this case.

(2) If $S^{n-1}1 = \frac{1}{\beta}$, then $S(S^{n-1}1,1)=(S1,1)$. Hence, there is no type $n$ interval for $S^n$, but (\ref{q:rec1}) and (\ref{q:rec2}) still hold for the other types. This gives the lemma for (b).

(3) Suppose $S^{n-1}1 > \frac{1}{\beta}$. Then $S(S^{n-1}1,1) = (S1,S^n1)$. Thus, the type ($n-1$) interval does not split and becomes a type $n$ interval for $S^n$. This implies that there is one type 1 interval less than there would be according to (\ref{q:rec1}). This gives the lemma in the even case.

\vskip .2cm
Now, let $n$ be odd and recall that the image of the type ($n-1$) interval under $S^{n-1}$ is $(0,S^{n-1}1)$. 

\vskip .2cm
(1*) Suppose $S^{n-1}1 < \frac{1}{\beta}$. Then $S(0, S^{n-1}1) = (S^n1,1)$. Thus, the type ($n-1$) does not split and becomes a type $n$ interval for $S^n$. Therefore, there is one type 0 interval less than there would be according to (\ref{q:rec1}). For all other values of $j$, the formulas (\ref{q:rec1}) and (\ref{q:rec2}) still hold.

(2*) If $S^{n-1}1 = \frac{1}{\beta}$, then $S(0,S^{n-1}1)=(0,1)$. Hence, there is no type $n$ interval for $S^n$, but (\ref{q:rec1}) and (\ref{q:rec2}) still hold for the other types. This gives the lemma for (b).

(3*) Suppose $S^{n-1}1 > \frac{1}{\beta}$. Then the type ($n-1$) interval splits into a type 0 interval and a type $n$ interval with image $(S^n1,1)$ under $S^n$. Hence, the recursions (\ref{q:rec1}) and (\ref{q:rec2}) still hold and this gives the lemma.
\end{proof}

\begin{figure}[ht]
\centering{\includegraphics{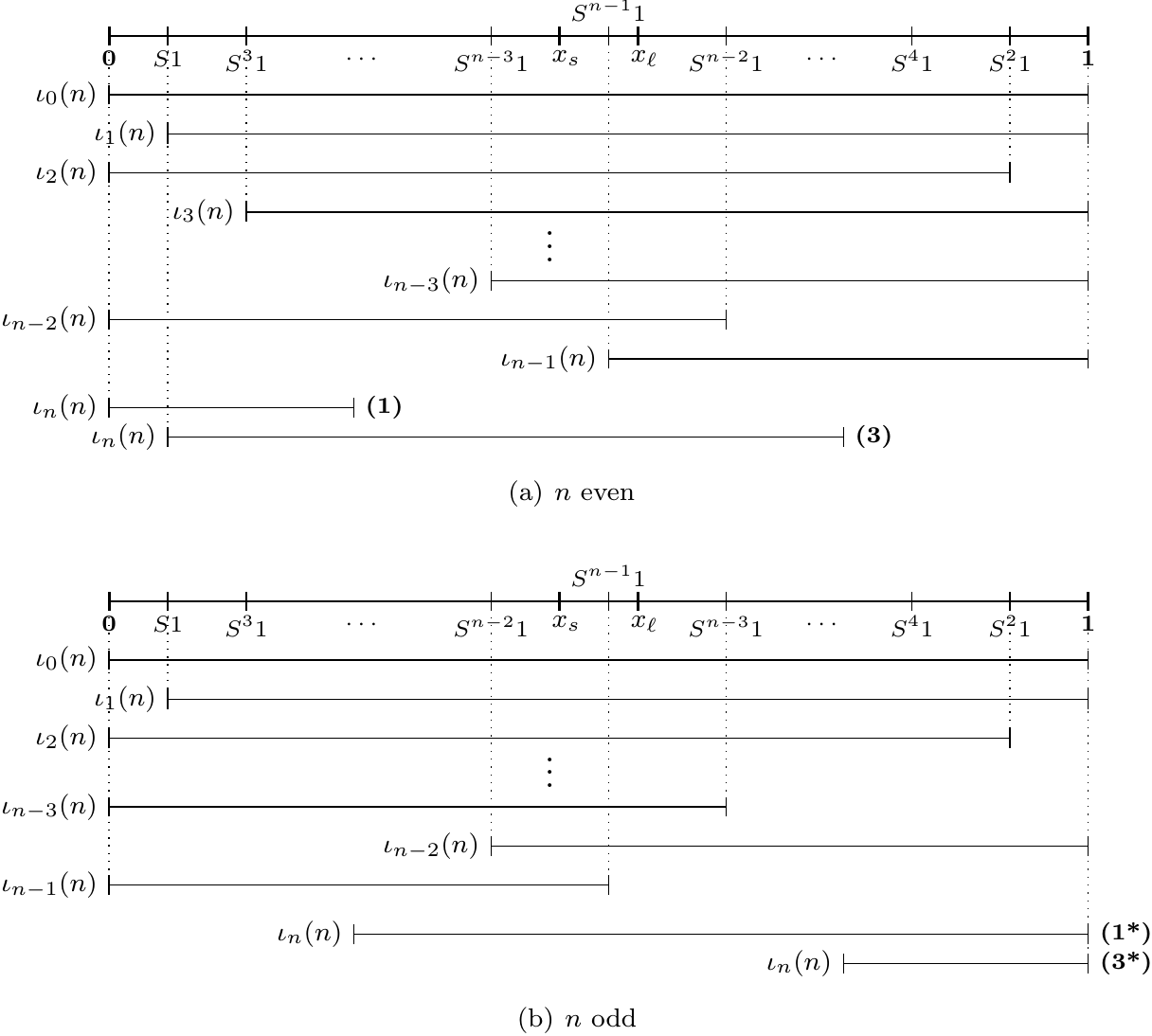}}
\caption{The images of the different types of intervals of monotonicity of $S^n$. We need to add the numbers $\iota_j(n)$ for the right $j$ to determine $\psi_n^-$. The (1), (3), (1*) and (3*) indicate the different cases from Lemma~\ref{l:cases}.}
\label{f:psi-}
\end{figure}

With this information we are ready to prove the next theorem.
\begin{theorem}\label{t:noniso}
For $n \ge 2$, let $\beta_n$ be the $n$-th multinacci number, i.e., the Pisot number with minimal polynomial $x^n-x^{n-1}-\cdots -x-1$. Fix $n$ and let $\beta_{n-1} < \beta < \beta_n$. Then the positive $\beta$-transformation $T$ and the negative $\beta$-transformation $S$ are not measurably isomorphic.
\end{theorem}

\begin{proof}
For the proof of this theorem we calculate some values in the range of $\psi_n^-$ and compare them to the information about $\psi^+_n$ obtained in Lemma~\ref{l:values} and Remark~\ref{r:values}. Our goal is to show that in all cases $\psi_n^+$ takes a value that does not occur for $\psi_n^-$ or the other way around. Just as for $T$, the points $S^k 1$ determine a partition of the interval $(0,1)$ into intervals on which $\psi^-_n$ is constant. To determine the values of $\psi_n^-$ we need to add the numbers $\iota_j(n)$ for the right $j$. For even $n$ the function $\psi_n^-$ takes its maximum on the interval $(S^{n-1}1,S^{n-2}1)$ and for odd $n$ the maximum is attained on the interval $(S^{n-2}1, S^{n-1}1)$, see Figure~\ref{f:psi-}.

First, let $n$ be an even number. We treat the three cases (1), (2) and (3) from Lemma~\ref{l:cases} separately.
\vskip .1cm

(1) If $S^{n-1}1< \frac{1}{\beta}$, then $0 < S^n 1 < x_s$. For $x \in (S^{n-1}1,S^{n-2}1)$, we have
\[ \psi_n^-(x) = \sum_{j=0}^{n-1} \iota_j(n) =2^{n-2} + 2 + \sum_{i=0}^{n-3} 3 \cdot 2^i = 2^{n-2} + 2+ \sum_{i=0}^{n-3} (2^{i+1} + 2^i) = 2^n-1.\]
Hence, according to Remark~\ref{r:values}(i), the maxima of $\psi_n^+$ and $\psi_n^-$ are equal. For $x \in (\max\{S^n 1,S^{n-3}1\}, S^{n-1}1)$ it holds that $\psi_n^-(x) = 2^n-1-\iota_{n-1}(n) = 2^n-3$. Note that $2^n-3$ is an odd number smaller than the maximum of $\psi_n^+$. On the other hand, on $(0, \min\{S^n 1, S1\})$,
\[ \psi_n^-(x) = \sum_{\stackrel{0 \le j \le n:}{j \text{ even}}} \iota_j(n) = 2^{n-2} + 1 + \sum_{\stackrel{0 \le j \le n-4:}{j \text{ even}}} 3\cdot 2^{n-j} = 2^{n-2} + 1 + 2^{n-2} -1=2^{n-1}.\]
Recall that by Lemma~\ref{l:values}, all the odd values in the range of $\psi^+_n$ other than the maximum must be smaller than the even values. Since $2^{n-1}$ is an even number and $2^{n-1} < 2^n-3$, either $\mu\big( I_n^+(2^n-3)\big)=0$ or $\mu\big(I_n^+(2^{n-1})\big)=0$. On the other hand, $\nu\big(I_n^-(2^{n-1})\big)>0$ and $\nu\big( I_n^-(2^n-3)\big)>0$. This implies that $T$ and $S$ cannot be isomorphic transformations.

\vskip .1cm
(2) If $S^{n-1}1 = \frac{1}{\beta}$, the maximal value of $\psi_n^-$ can be calculated just as for (a) and is equal to $2^n-1$. Then
\begin{eqnarray*}
\psi^-_n (x) \! \! &= \, \, 2^n-1 - \iota_{n-1}(n) = 2^n-3, & \text{if } x \in (S^{n-3}1,S^{n-1}1),\\
\psi^-_n (x) \! \! &= \, \, 2^n-1 - \iota_{n-2}(n) = 2^n-4, & \text{if } x \in (S^{n-2}1,S^{n-4}1),
\end{eqnarray*}
and $\psi_n^-(x) \le 2^n-4$ for all other $x \in [0,1]$. This means that $\psi_n^-$ does not take the value $2^n-2$, i.e., $\nu\big(I_n^-(2^n-2)\big)=0$. By Remark~\ref{r:values}(i), we know that $\mu\big(I_n^+(2^n-2)\big)>0$, showing that $T$ and $S$ cannot be isomorphic in this case.

\vskip .1cm

(3) If $S^{n-1} 1> \frac{1}{\beta}$, then $S^{n-2} 1 < S^n 1 < 1$. For $x \in (S^{n-1}1,S^{n-2}1)$, we have
\[ \psi_n^-(x) = 2^{n-2} + 3\cdot 2^{n-3}-1 + \sum_{i=0}^{n-4} 3 \cdot 2^i + 2 +1 =2^n-1.\]
Since $S^{n-2} 1 < S^n$, we have
\begin{eqnarray*}
\psi^-_n (x) \! \! &= \, \, 2^n-1 - \iota_{n-1}(n) = 2^n-3, & \text{if } x \in (S^{n-3}1,S^{n-1}1),\\
\psi^-_n (x) \! \! &= \, \, 2^n-1 - \iota_{n-2}(n) = 2^n-4, & \text{if } x \in (S^{n-2}1, \min\{S^n1,S^{n-4}1\}),
\end{eqnarray*}
and $\psi_n^-(x) \le 2^n-4$ for all other $x \in [0,1]$. Again $\psi_n^-$ does not take the value $2^n-2$ and thus by Remark~\ref{r:values}(i) $T$ and $S$ cannot be isomorphic. This finishes the proof for even $n$.

For odd $n$ the proof goes more or less along the same lines. Recall Figure~\ref{f:points}(b) and recall that the function $\psi_n^-$ takes its maximum on the interval $(S^{n-2}1,S^{n-1}1)$. Similar to above, it can be shown that also for odd $n$ this maximum is equal to $2^n-1$ in all cases. We will use Lemma~\ref{l:cases} to determine other values of $\psi_n^-$ as above and therefore we treat the three cases (1*), (2*) and (3*) separately.
\vskip .1cm

(1*) Suppose that $S^{n-1}1< \frac{1}{\beta}$. Then $0 < S^n 1 < S^{n-2}1$, which implies that
\begin{eqnarray*}
\psi^-_n (x) \! \! &= \, \, 2^n-1 - \iota_{n-1}(n) = 2^n-3, & \text{if } x \in (S^{n-1}1,S^{n-3}1),\\
\psi^-_n (x) \! \! &= \, \, 2^n-1 - \iota_{n-2}(n) = 2^n-4, & \text{if } x \in (\max\{S^n1,S^{n-4}1, S^{n-2}1\}),
\end{eqnarray*}
and $\psi_n^-(x) \le 2^n-4$ for all other $x \in [0,1]$. This means that $\psi_n^-$ does not take the value $2^n-2$, i.e., $\nu\big(I_n^-(2^n-2)\big)=0$. By Remark~\ref{r:values}(i), we know that $\mu\big(I_n^+(2^n-2)\big)>0$, showing that $T$ and $S$ cannot be isomorphic in this case.

\vskip .1cm
(2*) If $S^{n-1}1 = \frac{1}{\beta}$, then
\begin{eqnarray*}
\psi^-_n (x) \! \! &= \, \, 2^n-1 - \iota_{n-1}(n) = 2^n-3, & \text{if } x \in (S^{n-1}1, S^{n-3}1),\\
\psi^-_n (x) \! \! &= \, \, 2^n-1 - \iota_{n-2}(n) = 2^n-4, & \text{if } x \in (S^{n-4}1,S^{n-2}1),
\end{eqnarray*}
and $\psi_n^-(x) \le 2^n-4$ for all other $x \in [0,1]$. Again $\psi_n^-$ does not take the value $2^n-2$ and thus by Remark~\ref{r:values}(i) $T$ and $S$ cannot be isomorphic.
\vskip .1cm

(3*) If $S^{n-1} 1> \frac{1}{\beta}$, then $x_{\ell} < S^n 1 < 1$. According to Remark~\ref{r:values}, the maxima of $\psi_n^+$ and $\psi_n^-$ are equal. For $x \in (S^{n-1}1,\min\{S^n1 , S^{n-3}1\})$, we have $ \psi_n^-(x) = 2^n-1-\iota_{n-1}(n)=2^n-3$, which is an odd number smaller than the maximum of $\psi_n^+$. On the other hand, on $(\max\{S^n 1, S^21\},1)$,
\[ \psi_n^-(x) = \iota_0(n)+\sum_{\stackrel{0 \le j \le n:}{j \text{ odd}}} \iota_j(n) = 2^{n-2} + 3\cdot 2^{n-3} + \cdots + 3 +1 = 2^{n-2} + 2^{n-1},\]
which is an even number and bigger than $2^n-3$. By Remark~\ref{r:values}(ii), either $\mu\big(I_n^+(2^{n-2}+2^{n-1})\big)=0$ or $\mu\big( I_n^+(2^n-3)\big)=0$, while $\nu\big(I_n^-(2^{n-2}+2^{n-1})\big)>0$ and $\nu\big( I_n^-(2^n-3)\big)>0$. This implies that $T$ and $S$ cannot be isomorphic transformations.
\end{proof}

\begin{remark}{\rm
Note that the previous theorem gives an illustration of the well known fact that two one-sided Markov shifts on a finite alphabet with the same topological entropy do not necessarily have to be isomorphic. For all Pisot numbers $\beta$ the maps $T$ and $S$ are isomorphic to a one-sided finite alphabet Markov shift with topological entropy equal to $h_{\mu}(T)=\log \beta$. Since Pisot numbers lie dense in the interval $(\beta_2, 2)$, there are many Pisot numbers in this interval for which the maps $T$ and $S$ are not isomorphic. Hence, the same holds for the associated Markov shifts.
}\end{remark}

\begin{remark}{\rm
One can still wonder how different the dynamics of positive slope and negative slope piecewise linear maps really is. For example, in case $1 <\beta <2$ is not a multinacci number, does there exists any piecewise linear map with constant slope $\beta$ that is isomorphic to the map $S$? If the map $S$ is a piecewise linear Markov map for the $\beta$ in question, then one can always construct an isomorphic piecewise linear map with constant slope $\beta$ based on the Markov partition and the corresponding (0,1)-matrix. But what if $S$ is not a piecewise linear Markov map?
}\end{remark}

\begin{remark}{\rm
Theorems~\ref{t:multinacci} and~\ref{t:noniso} will most likely extend to $\beta$-transformations with more than 2 digits, i.e., with $\beta >2$. We then define $T_{\beta} x = \beta x - \lfloor \beta x \rfloor$ and $S_{\beta} x = 1+\lfloor \beta x \rfloor -\beta x$. The two maps $T_{\beta}$ and $S_{\beta}$ are probably isomorphic if and only if $\beta$ is an integer or a `generalised multinacci number'. The generalised multinacci numbers are the numbers satisfying $1 = \frac{\lfloor \beta \rfloor}{\beta}  + \cdots + \frac{\lfloor \beta \rfloor}{\beta^n} + \frac{j}{\beta^{n+1}}$ with $ j \in \{0,1, \ldots, \lfloor \beta \rfloor \}$ and $n \ge 1$. For $1 < \beta < 2$ this leaves only the multinacci numbers.
}\end{remark}

\end{document}